\documentclass[11pt, letterpaper]{amsart}
\usepackage{appendix}
\usepackage{xcolor}
\usepackage[utf8]{inputenc}
\usepackage[colorlinks=true,hyperindex,pagebackref,linktocpage=true]{hyperref}
\usepackage{cleveref}
\usepackage{amsfonts}
\usepackage{amssymb}
\usepackage{tikz-cd}
\usepackage[T1]{fontenc}

\hypersetup{
	colorlinks,
	linkcolor={violet!60!black},
	citecolor={cyan!60!black},
	urlcolor={orange!60!black}
}
\usepackage{stmaryrd}
\usepackage{mathrsfs}  
\usepackage{varwidth}

\theoremstyle{plain}
\newtheorem{theorem}{Theorem}[section]
\newtheorem{proposition}[theorem]{Proposition}
\newtheorem{lemma}[theorem]{Lemma}
\newtheorem{corollary}[theorem]{Corollary}

\theoremstyle{definition}
\newtheorem{definition}[theorem]{Definition}

\newtheorem{remark}[theorem]{Remark}

\newcommand{\nc}{\newcommand}
\nc{\on}{\operatorname}

\nc{\Q}{\mathbb{Q}}
\nc{\Z}{\mathbb{Z}}
\nc{\cl}{\mathrm{cl}}

\nc{\fraka}{{\mathfrak a}} \nc{\bba}{{\mathbf a}}
\nc{\frakb}{{\mathfrak b}}
\nc{\frakc}{{\mathfrak c}}
\nc{\frakd}{{\mathfrak d}}
\nc{\frake}{{\mathfrak e}}
\nc{\frakf}{{\mathfrak f}}
\nc{\frakg}{{\mathfrak g}}
\nc{\frakh}{{\mathfrak h}}
\nc{\fraki}{{\mathfrak i}}
\nc{\frakj}{{\mathfrak j}}
\nc{\frakk}{{\mathfrak k}}
\nc{\frakl}{{\mathfrak l}}
\nc{\frakm}{{\mathfrak m}}
\nc{\frakn}{{\mathfrak n}}
\nc{\frako}{{\mathfrak o}}
\nc{\frakp}{{\mathfrak p}}
\nc{\frakq}{{\mathfrak q}}
\nc{\frakr}{{\mathfrak r}}
\nc{\fraks}{{\mathfrak s}}
\nc{\frakt}{{\mathfrak t}}
\nc{\fraku}{{\mathfrak u}}
\nc{\frakv}{{\mathfrak v}}
\nc{\frakw}{{\mathfrak w}}
\nc{\frakx}{{\mathfrak x}}
\nc{\fraky}{{\mathfrak y}}
\nc{\frakz}{{\mathfrak z}}
\nc{\frakA}{{\mathfrak A}}
\nc{\frakB}{{\mathfrak B}}
\nc{\frakC}{{\mathfrak C}}
\nc{\frakD}{{\mathfrak D}}
\nc{\frakE}{{\mathfrak E}}
\nc{\frakF}{{\mathfrak F}}
\nc{\frakG}{{\mathfrak G}}
\nc{\frakH}{{\mathfrak H}}
\nc{\frakI}{{\mathfrak I}}
\nc{\frakJ}{{\mathfrak J}}
\nc{\frakK}{{\mathfrak K}}
\nc{\frakL}{{\mathfrak L}}
\nc{\frakM}{{\mathfrak M}}
\nc{\frakN}{{\mathfrak N}}
\nc{\frakO}{{\mathfrak O}}
\nc{\frakP}{{\mathfrak P}}
\nc{\frakQ}{{\mathfrak Q}}
\nc{\frakR}{{\mathfrak R}}
\nc{\frakS}{{\mathfrak S}}
\nc{\frakT}{{\mathfrak T}}
\nc{\frakU}{{\mathfrak U}}
\nc{\frakV}{{\mathfrak V}}
\nc{\frakW}{{\mathfrak W}}
\nc{\frakX}{{\mathfrak X}}
\nc{\frakY}{{\mathfrak Y}}
\nc{\frakZ}{{\mathfrak Z}}
\nc{\bbA}{{\mathbb A}}
\nc{\bbB}{{\mathbb B}}
\nc{\bbC}{{\mathbb C}}
\nc{\bbD}{{\mathbb D}}
\nc{\bbE}{{\mathbb E}}
\nc{\bbF}{{\mathbb F}} \nc{\bbf}{{\mathbf f}}
\nc{\bbG}{{\mathbb G}}
\nc{\bbH}{{\mathbb H}}
\nc{\bbI}{{\mathbb I}}
\nc{\bbJ}{{\mathbb J}}
\nc{\bbK}{{\mathbb K}}
\nc{\bbL}{{\mathbb L}}
\nc{\bbM}{{\mathbb M}}
\nc{\bbN}{{\mathbb N}}
\nc{\bbO}{{\mathbb O}}
\nc{\bbP}{{\mathbb P}}
\nc{\bbQ}{{\mathbb Q}}
\nc{\bbR}{{\mathbb R}}
\nc{\bbS}{{\mathbb S}}
\nc{\bbT}{{\mathbb T}}
\nc{\bbU}{{\mathbb U}}
\nc{\bbV}{{\mathbb V}}
\nc{\bbW}{{\mathbb W}}
\nc{\bbX}{{\mathbb X}}
\nc{\bbY}{{\mathbb Y}}
\nc{\bbZ}{{\mathbb Z}}
\nc{\calA}{{\mathcal A}}
\nc{\calB}{{\mathcal B}}
\nc{\calC}{{\mathcal C}}
\nc{\calD}{{\mathcal D}}
\nc{\calE}{{\mathcal E}}
\nc{\calF}{{\mathcal F}}
\nc{\calG}{{\mathcal G}}
\nc{\calH}{{\mathcal H}}
\nc{\calI}{{\mathcal I}}
\nc{\calJ}{{\mathcal J}}
\nc{\calK}{{\mathcal K}}
\nc{\calL}{{\mathcal L}}
\nc{\calM}{{\mathcal M}}
\nc{\calN}{{\mathcal N}}
\nc{\calO}{{\mathcal O}}
\nc{\calP}{{\mathcal P}}
\nc{\calQ}{{\mathcal Q}}
\nc{\calR}{{\mathcal R}}
\nc{\calS}{{\mathcal S}}
\nc{\calT}{{\mathcal T}}
\nc{\calU}{{\mathcal U}}
\nc{\calV}{{\mathcal V}}
\nc{\calW}{{\mathcal W}}
\nc{\calX}{{\mathcal X}}
\nc{\calY}{{\mathcal Y}}
\nc{\calZ}{{\mathcal Z}}

\nc{\scrA}{{\mathscr A}}
\nc{\scrB}{{\mathscr B}}
\nc{\scrC}{{\mathscr C}}
\nc{\scrD}{{\mathscr D}}
\nc{\scrE}{{\mathscr E}}
\nc{\scrF}{{\mathscr F}}
\nc{\scrG}{{\mathscr G}}
\nc{\scrH}{{\mathscr H}}
\nc{\scrI}{{\mathscr J}}
\nc{\scrJ}{{\mathscr I}}
\nc{\scrK}{{\mathscr K}}
\nc{\scrL}{{\mathscr L}}
\nc{\scrM}{{\mathscr M}}
\nc{\scrN}{{\mathscr N}}
\nc{\scrO}{{\mathscr O}}
\nc{\scrP}{{\mathscr P}}
\nc{\scrQ}{{\mathscr Q}}
\nc{\scrR}{{\mathscr R}}
\nc{\D}{{\on{D}}}
\nc{\Div}{{\on{Div}}}
\nc{\Perv}{{\on{Perv}}}

\nc{\bnu}{{\bar{ \nu}}}
\nc{\olO}{\bar{\calO}}

\nc{\al}{{\alpha}} 
\nc{\be}{{\beta}}
\nc{\ga}{{\gamma}} \nc{\Ga}{{\Gamma}}
\nc{\hGa}{\hat{\Gamma}}
\nc{\ve}{{\varepsilon}} 
\nc{\la}{{\lambda}} \nc{\La}{{\Lambda}}
\nc{\om}{\omega} \nc{\Om}{\Omega} 
\nc{\sig}{{\sigma}} \nc{\Sig}{{\Sigma}}
\nc{\dR}{{\mathrm{dR}}}
\nc{\Perf}{{\mathrm{Perf}}}
\nc{\Gm}{{\mathbb{G}_m}}
\nc{\colim}{{\on{colim}}}

\makeatletter
\DeclareFontEncoding{LS1}{}{}
\DeclareFontSubstitution{LS1}{stix}{m}{n}
\DeclareMathAlphabet{\rhomalpha}{LS1}{stixscr}{m}{n}
\makeatother

\nc{\Spa}{\on{{Spa}}}
\nc{\Spd}{\on{{Spd}}}
\nc{\tnb}{\psi_{\rm tame}}
\nc{\oM}{\overline{{M}}}
\nc{\op}{{\on{op}}}
\nc{\ad}{{\on{ad}}}
\nc{\alg}{{\on{alg}}}
\nc{\Ad}{{\on{Ad}}}
\nc{\Adm}{{\on{Adm}}} \nc{\aff}{{\on{af}}}
\nc{\Aut}{{\on{Aut}}}
\nc{\Bun}{{\on{Bun}}}
\nc{\cha}{{\on{char}}}
\nc{\der}{{\on{der}}}
\nc{\Der}{{\on{Der}}}
\nc{\diag}{{\on{diag}}}
\nc{\End}{{\on{End}}}
\nc{\Fl}{{\calF\!\ell}}
\nc{\Tr}{{\on{Transp}}}
\nc{\TR}{{\calT\!\calR}}
\nc{\Gal}{{\on{Gal}}}
\nc{\Gr}{{\on{Gr}}}
\nc{\Hk}{{\on{Hk}}}
\nc{\rH}{{\on{H}}}
\nc{\Hom}{{\on{Hom}}}
\nc{\IC}{{\on{IC}}}
\nc{\id}{{\on{id}}}
\nc{\Id}{{\on{Id}}}
\nc{\ind}{{\on{ind}}}
\nc{\Ind}{{\on{Ind}}}
\nc{\Lie}{{\on{Lie}}}
\nc{\Pic}{{\on{Pic}}}
\nc{\pr}{{\on{pr}}}
\nc{\Res}{{\on{Res}}}
\nc{\res}{{\on{res}}} \nc{\Sat}{{\on{Sat}}}
\nc{\spc}{{\on{sc}}}
\nc{\drv}{{\on{der}}}
\nc{\sgn}{{\on{sgn}}}
\nc{\Spec}{{\on{Spec}}}\nc{\Spf}{\on{Spf}} 
\nc{\Sph}{\on{Sph}}
\nc{\St}{{\on{St}}}
\nc{\tr}{{\on{tr}}}
\nc{\Mod}{{\mathrm{-Mod}}}
\nc{\Hilb}{{\on{Hilb}}} 
\nc{\Ext}{{\on{Ext}}} 
\nc{\vs}{{\on{Vec}}}
\nc{\ev}{{\on{ev}}}
\nc{\nO}{{\breve{\calO}}}
\nc{\tS}{{\tilde{S}}}
\nc{\spe}{{\on{sp}}}
\nc{\loc}{{\on{loc}}}
\nc{\pre}{{\on{pre}}}

\nc{\co}{\colon}
\nc{\dia}{{\diamondsuit}}

\nc{\nscrR}{{\mathscr{R}^{\on{nr}}}}

\nc{\GL}{{\on{GL}}}

\nc{\Gl}{\on{Gl}} 
\nc{\GSp}{{\on{GSp}}}
\nc{\gl}{{\frakg\frakl}}
\nc{\SL}{{\on{SL}}} 
\nc{\SU}{{\on{SU}}} 
\nc{\SO}{{\on{SO}}}
\nc{\PGL}{{\on{PGL}}}

\nc{\Conv}{{\on{Conv}}}
\nc{\Rep}{{\on{Rep}}}
\nc{\Dom}{{\on{Dom}}}
\nc{\red}{{\on{red}}}
\nc{\act}{{\on{act}}}
\nc{\nr}{{\on{nr}}}
\nc{\ctf}{{\on{ctf}}}

\nc{\str}{{\on{-}}} 
\nc{\os}{{\bar{s}}}
\nc{\oeta}{{\bar{\eta}}}

\nc{\et}{\textup{\'et}}

\nc{\hookto}{\hookrightarrow}
\nc{\longto}{\longrightarrow}
\nc{\leftto}{\leftarrow}
\nc{\onto}{\twoheadrightarrow}
\nc{\lonto}{\twoheadleftarrow}

\nc{\pot}[1]{ [\hspace{-0,5mm}[ {#1} ]\hspace{-0,5mm}] }
\nc{\rpot}[1]{ (\hspace{-0,7mm}( {#1} )\hspace{-0,7mm}) }

\numberwithin{equation}{section}

\begin{document}
	
	\title{Tubular neighborhoods of local models}
	
	\author[I. Gleason, J. Louren\c{c}o]{Ian Gleason, Jo\~ao Louren\c{c}o}

	\address{Mathematisches Institut der Universit\"at Bonn, Endenicher Allee 60, Bonn, Germany}
	\email{igleason@uni-bonn.de}
	
		\address{Universität Münster, Einsteinstrasse 62, Münster, Germany}
	\email{j.lourenco@uni-muenster.de}
	
	\begin{abstract}
		We show that the v-sheaf local models of \cite{SW20} are unibranch. In particular, this proves that the scheme-theoretic local models defined in \cite{AGLR22} are always normal with reduced special fiber, thereby establishing the remaining cases of the geometric part of the Scholze--Weinstein conjecture when $p \leq 3$. Our methods are general, topological and simplify those of \cite{Zhu14} for tamely ramified groups in positive characteristic. As a technical input, we generalize a comparison theorem of nearby cycles of \cite{Hub96} to the v-sheaf setup.
	\end{abstract}

	\maketitle
	\tableofcontents
	
	\section{Introduction}
	
	Local models were introduced in the nineties to study the singularities of Shimura varieties, namely in the works of Chai--Norman \cite{CN90}, de Jong \cite{dJ93} and Deligne--Pappas \cite{DP94}, and have found various applications. The theory was systematized in the book of Rapoport--Zink \cite{RZ96}, via linear algebraic moduli problems. Later, it underwent a significant transformation when Görtz \cite{Gor01,Gor03} embedded their special fibers in certain infinite-dimensional flag varieties. This was subsequently exploited by Faltings, Pappas, Rapoport and Zhu \cite{Fal03,PR08,Zhu14,PZ13} to great effect. More recently, Scholze--Weinstein \cite{SW20} proposed a fully functorial avenue to study local models in mixed characteristic via perfectoid geometry. This program was pursued in \cite{AGLR22}. We use the following notation in the paper.
	
	\begin{definition}\label{defn_LM}
		Let $F$ be a local field, $O$ its ring of integers, and $k$ its residue field. Let $G$ be a reductive connected $F$-group, $\calG$ a parahoric $O$-model of $G$, $\mu$ a geometric conjugacy class of cocharacters, and $E$ its reflex field. We denote by $\calM_{\calG,\mu}$ the v-sheaf given as the closure of $\Gr_{G,\mu} \subset \Gr_{G,E}$ inside the Beilinson--Drinfeld Grassmannian $\Gr_{\calG,O_E}$ of \cite{SW20,FS21}. If $F$ is of positive characteristic or $\mu$ is minuscule, we denote by $\calM_{\calG,\mu}^{\on{sch}}$ the canonical weakly normal\footnote{Recall that a scheme is weakly normal if every finite, birational, universally homeomorphic morphism with reduced source is an isomorphism.} proper $O_E$-scheme representing $\calM_{\calG,\mu}$. 
	\end{definition}
	
	If $F$ is $p$-adic and $\mu$ is minuscule, then there is a unique weakly normal scheme $\calM_{\calG,\mu}^{\on{sch}}$ representing $\calM_{\calG,\mu}$ by \cite[Theorem 1.1]{AGLR22}. If $F$ has positive characteristic, then $\calM_{\calG,\mu}^{\on{sch}}$ is defined as the weakly normal scheme representing $\calM_{\calG,\mu}$ whose generic fiber is a Schubert variety in positive characteristic.
	
	Historically, an important goal has been to show that the special fiber of $\calM_{\calG,\mu}^{\on{sch}}$ is reduced, see \cite{PRS13} for various ad hoc definitions of $\calM_{\calG,\mu}^{\on{sch}}$.
	Görtz \cite{Gor01,Gor03} proved this for split classical groups of PEL type via so-called straightening laws, but the proof does not extend to the general case. An important development is due to Pappas--Rapoport \cite{PR08}, who formulated the coherence conjecture to address reducedness via coherent cohomology of ample line bundles. Finally, Zhu \cite{Zhu14} proved the conjecture for tame groups, by translating the problem to equicharacteristic, and constructing a global Frobenius splitting of the local model compatibly with the special fiber. Recall that a scheme $X$ in characteristic $p$ is Frobenius split if the $\calO_X$-module homomorphism $\calO_X \to \varphi_*\calO_X$ given by Frobenius splits. Note that Frobenius split schemes are necessarily reduced. We refer to \cite{BK07,BS13} for proper introductions to the subject.
	
	Most modern results in the literature concerning the reduced structure of the local models rely on \cite{Zhu14} through reductions and comparisons. Interestingly, the heart of Zhu's proof lies in characteristic $p$. This becomes problematic in the perfectoid perspective of \cite{SW20,AGLR22}, because it is not clear how to work with Frobenius splitting techniques in this context. Fortunately, we have the following crucial fact that allow us to bypass them:
	
	\begin{lemma}\label{lem_unibranch_reduced_sp_fib}
		If $\calM_{\calG,\mu}^{\on{sch}}$ is unibranch, then its special fiber is reduced.
	\end{lemma}
	
	\begin{proof}
		This is \cite[Lemma 7.26]{AGLR22}, but we repeat the crux of the argument for future reference and the reader's convenience. We know already that the perfection of the special fiber of $\calM_{\calG,\mu}^{\on{sch}}$ equals the so-called $\mu$-admissible locus $\calA_{\calG,\mu}$, see \cite[Theorem 6.16]{AGLR22}. In particular, the union of maximal orbits defines a smooth open of $\calM_{\calG,\mu}^{\on{sch}}$ with dense geometric fibers, see \cite[Corollary 2.14]{Ric16} and \cite[Equation (7.44)]{AGLR22}, so the special fiber satisfies Serre's condition $R_0$. As $\calM_{\calG,\mu}^{\on{sch}}$ is weakly normal per definition, the unibranch assumption in the statement implies normality already. In particular, the special fiber is $S_1$ by the Serre criterion for normality plus flatness. The Serre criterion for reducedness yields our claim.
	\end{proof}
	
	The unibranch property is already amenable to a formulation in terms of perfectoids, because it is topological in nature. Indeed, it suffices to know that the tubular neighborhoods of \cite[Definition 4.38]{Gle22} at all the closed points of the special fiber are connected. In turn, being connected is a cohomological invariant and it is natural to expect that a deeper study of the nearby cycles initiated in \cite{AGLR22} would yield the result. Over a $p$-adic field and for non-minuscule coweights $\mu$, the v-sheaf $\calM_{\calG,\mu}$ does not come from a scheme. Nevertheless, $\calM_{\calG,\mu}$ is a kimberlite, the v-sheaf analogue of a formal scheme, and being unibranch (or topologically normal) still makes sense in this context, see \cite[Definition 4.52]{Gle22}. Our main result is:
	\begin{theorem}\label{thm_normal_LM}
		The kimberlite $\calM_{\calG,\mu}$ is unibranch for all pairs $(\calG,\mu)$ from \Cref{defn_LM}. 
	\end{theorem}
	
	As an immediate corollary, we get the geometric part of the Scholze--Weinstein conjecture in full generality, removing certain exceptions found in \cite[Theorem 7.23]{AGLR22} for $p\leq 3$.
	\begin{corollary}\label{cor_CM_LM_sch}
		If $F=k\rpot{t}$ or $\mu$ is minuscule, the underlying scheme $\calM_{\calG,\mu}^{\on{sch}}$ is normal with reduced special fiber. If $p>2$ or $\Phi_G$ is reduced, $\calM_{\calG,\mu}^{\on{sch}}$ is Cohen--Macaulay with Frobenius split special fiber.
	\end{corollary}
	
	\begin{proof}The first sentence follows from \Cref{thm_normal_LM} and \Cref{lem_unibranch_reduced_sp_fib}. In the last sentence, we must exclude the case $p=2$ and $\Phi_G$ non-reduced as in \cite[Assumption 1.9]{AGLR22}. Our goal following \cite[Conjecture 7.25]{AGLR22} is to identify the special fiber of $\calM_{\calG,\mu}^{\on{sch}}$ with the canonical deperfection $\calA_{\calG,\mu}^{\mathrm{sch}}$ of the $\mu$-admissible locus, which follows from \cite[Theorem 4.1, Corollary 5.9]{FHLR22} in equicharacteristic and \cite[Theorem 3.16]{AGLR22} for minuscule $\mu$. 
	\end{proof}
	\begin{remark}
		Another direct corollary of \Cref{thm_normal_LM} is that moduli of $p$-adic shtukas in the sense of \cite[Definition 2.27]{Gle21} or the more restrictive \cite[Definition 3.2.1]{PR21} are also unibranch. 
	\end{remark}
	
	Let us explain the strategy behind \Cref{thm_normal_LM}. The main idea is that the stalks of the first non-trivial cohomology sheaf of the nearby cycles $R\Psi(\on{IC}_\mu)$ detects the number of connected components of the tubular neighborhoods. In this way, a disconnection of a tubular neighborhood produces a disconnection \'etale locally. This however involves a comparison between the analytic nearby cycles defined in \cite{Sch17} and the formal nearby cycles defined in \cite{Gle22}, see \Cref{thmnearbycycles}. In general, such a result only holds if the analytic nearby cycles are already algebraic, see \Cref{defialgebraicity}. Fortunately, algebraicity of $R\Psi(\on{IC}_\mu)$ was proved in \cite{AGLR22}. We expect this theorem to find broader applications in the study of integral models of local Shimura varieties and moduli of $p$-adic shtukas.
	
	If $\calG$ is Iwahori (enough for our purposes by \Cref{lem_iwahori_conn_fibers} and \cite[Lemma 5.26]{Gle22}), one uses the theory of Wakimoto filtrations in mixed characteristic developed in \cite{ALWY22} to bound the dimension of stalks of $R\Psi(\on{IC}_\mu)$ along orbits in codimension at most $1$ and to prove directly that $\calM_{\calG,\mu}$ is unibranch away from a subset of codimension $2$. To prove unibranchness in deeper strata, one uses the perversity of $R\Psi(\on{IC}_\mu)$ due to \cite{ALWY22} and a combinatorial argument. To deal with codimension $2$ subsets in the special fiber, we observe that $\calA_{\calG,\mu}$ is perfectly $S_2$ (i.e. the perfection of a scheme that satisfies Serre's condition $S_2$).\footnote{Under the conditions of \Cref{cor_CM_LM_sch}, we know that $\calA_{\calG,\mu}^{\mathrm{sch}}$ is Cohen--Macaulay, compare with \cite{HR22}. If $F$ is $p$-adic and $\mu$ is not minuscule, it is reasonable to still expect $\calA^{\mathrm{sch}}_{\calG,\mu}$ to be Cohen--Macaulay.} The $S_2$ property can be neatly expressed combinatorially due to \cite{HH94b}, so it reduces to positive characteristic.
	Morally, the $S_2$ property implies that the normalization map has to be an isomorphism. Unfortunately, that map is not available for kimberlites, so we argue more carefully. We exploit the $S_2$ property and unibranchness on codimension $1$ strata to prove that an \'etale-formal local disconnection of the generic fiber forces a large open subset to specialize to a small stratum contradicting the perversity of $R\Psi(\on{IC}_\mu)$.  
	
	In positive characteristic, we reprove normality for tame groups using $\bbG_{m,k}$-actions, following techniques of Le--Levin--Le Hung--Morra \cite{LLHLM20}, who study the unibranch property for the more singular crystalline local models. The strategy for proving that the local model is unibranch in this case was therefore known to the authors of \cite{LLHLM20}, but they did not seem to know of \Cref{lem_unibranch_reduced_sp_fib}.
	
	\subsection{Acknowledgements} We thank Johannes Anschütz, Ana Caraiani, K\k{e}stutis \v{C}esnavi\v{c}ius, Ulrich Görtz, Thomas Haines, Xuhua He, Daniel Le, Arthur-César Le Bras, Bao Viet Le Hung, Brandon Levin, Dong Gyu Lim, Stefano Morra, Timo Richarz, Mafalda Santos, Peter Scholze, Sug Woo Shin, Zhiyou Wu, Jize Yu, Xinwen Zhu for helpful discussions related to the project. We also owe special thanks to the following people: A.-C. Le Bras for encouraging us to finally read \cite{LLHLM20}; Tom Haines for drawing our attention to \cite[Proposition 8.7]{Hai04}, and supplying an enhancement to the non-minuscule case, see \Cref{rem_irred_x_explicit}; and Peter Scholze for detecting a serious gap in one of our drafts.
	
	This paper was written during stays at Imperial College, Orsay, Max-Planck-Institut für Mathematik, Universität Bonn, and we are thankful for the hospitality of these institutions. The project has received funding (I.G.~via Peter Scholze) from the Leibniz-Preis, (J.L.~via Ana Caraiani) from the European Research Council under the European Union’s Horizon 2020 research and innovation program (grant agreement nº 804176), and (J.L.) from the Max-Planck-Institut für Mathematik.
	
	\section{New proof of Zhu's theorem}
	
	In this section, we establish a particular case of \Cref{thm_normal_LM}, originally due to \cite{Zhu14}, via global methods specific to positive characteristic. Let $k$ be algebraically closed of characteristic $p$ and $G$ a connected reductive $\bbG_{m,k}$-group that splits over the finite étale cover $\bbG_{m,k} \to \bbG_{m,k}$ given by raising to the $e$-th power, with $(e,p)=1$. Let $\calG$ be the $\bbA^1_k$-model of $G$ built out of a parahoric $k\pot{t}$-model of $G_{k\rpot{t}}$ and $G$ via Beauville--Laszlo descent. We regard the Beilinson--Drinfeld Grassmannian $\Gr_{\calG}$ as being an ind-(perfect scheme) defined over the perfection $\bbA^{1,\mathrm{pf}}_k$ of the degree $e$ cover $\bbA^1_k \to \bbA^1_k$ that splits $G$ away from the origin, the local model $\calM_{\calG,\mu}$ given as the closure of $\Gr_{G,\mu}$ inside $\Gr_{\calG}$, and finally $\calM^{\on{sch}}_{\calG,\mu}$ as the canonical weakly normal deperfection of $\calM_{\calG,\mu}$ finitely presented over $\bbA^1_k$. Apologies to the reader are in order for deviating from the notation in \Cref{defn_LM}, which referred to the associated v-sheaves over a complete local ring, but it is not hard to reconcile both perspectives.
	
	We aim to reprove \cite[Theorem 3]{Zhu14}. Thanks to \Cref{lem_unibranch_reduced_sp_fib}, this can proceed along the lines of \cite[Section 3]{LLHLM20}.

	\begin{theorem}[\cite{Zhu14}]\label{thm_zhu_coh}
		The flat projective scheme $\calM_{\calG,\mu}^{\on{sch}}$ is normal with Frobenius split, reduced fiber over $0$.
	\end{theorem}

	A special feature of tame groups in equicharacteristic is that $\Gr_{\calG}$ carries the so-called rotation $\bbG_{m,k}^{\on{pf}}$-action which lifts the $e$-th power of the natural one on the base $\bbA^{1,\mathrm{pf}}_{k}$, see \cite[Section 5]{Zhu14}. 
	Also, we can regard a maximal $F$-split torus $S\subset G$ as defined over $k$ up to unique isomorphism, and we get a natural action of $S^{\on{pf}}$ on $\Gr_{\calG}$ linear over $\bbA^{1,\mathrm{pf}}_{k}$. For any coweight $\chi$ of $S \times \bbG_{m,k}$, the induced $\bbG_{m,k}^{\on{pf}}$-action on $\Gr_{\calG}$ is Zariski locally linearizable in the sense of \cite{Ric19a}. This is seen by reduction to $\GL_n$, where we reason via lattices as in \cite[Lemma 3.3]{HR21}.  Now, the attractor $\Gr_{\calG}^+$ exists by \cite[Theorem 1.8]{Ric19a} and \cite[Theorem 2.1]{HR21}, and is representable by a disjoint union of locally closed sub-ind-schemes. By compactifying $\Gr_{\calG}$ to $\bbP^{1,\mathrm{pf}}_k$ (simply extend $\calG$ further to a parahoric $\bbP^1_k$-group scheme, see \cite[Définition 4.2.8]{Lou19}), we see by \cite[Lemma 1.11]{Ric19a} that the attractor $\Gr_{\calG}^+$ maps surjectively to $\Gr_{\calG}$ if $\chi$ is contracting on $\bbA^{1,\mathrm{pf}}_{k}$. We denote by $\Fl_{\calG}$ the fiber of $\Gr_{\calG}$ over $0$.

	\begin{lemma}\label{lemma_chi_w_existence}
		For every $S$-fixed point $w \in \Fl_{\calG}$, there exists $\chi_w \colon \bbG_{m,k} \to S \times \bbG_{m,k}$ such that $w \in \Fl_{\calG}^0$ is isolated and the connected component of $\Fl_{\calG}^+$ containing $w$ is open in $\Fl_{\calG}$.
	\end{lemma}
	
	\begin{proof}
		Recall that $S\times \bbG_{m,k}$ supports a Kac--Moody root system, see \cite[Définition 4.2.1, Lemme 4.2.2]{Lou19}. Let $\chi\colon \bbG_{m,k} \to S \times \bbG_{m,k}$ be a coweight lying in the $w$-conjugate of the anti-dominant facet of type corresponding to $\calG$. Then, the connected component of $\Fl_{\calG}^+$ containing $w$ equals the open left translate $w\cdot L^{--}\calG \cdot e \subset \Fl_{\calG}$ of the big cell, see \cite[Corollaire 4.2.11]{Lou19}.
	\end{proof}
	
	Notice that $\chi_w$ is anti-dominant for the Kac--Moody root system, meaning it acts on the variable $t_1$ defining the flag variety of $\Fl_{\calG}$ by negative powers, whereas we expect $\chi_w$ to contract the affine line $\bbA^{1,\mathrm{pf}}_k$ that serves as base of the local model. The change of sign occuring here is explained by the fact that the definition of $\Gr_{\calG}$ involves an auxiliary formal variable $t_2$, so that $r^{-1}t_1-t_2$ and $t_1-rt_2$ define the same Cartier divisor. 
	
	The next task is to globalize the open set of the previous lemma. In the case of the Iwahori model of ${\rm GL}_n$, the desired open neighborhood is constructed explicitly in \cite[Lemma 3.2.7]{LLHLM20}. Instead, we provide an abstract argument.

	\begin{lemma}\label{lemma_mult_action_global_grass}
		Let $\chi_w$ be as in \Cref{lemma_chi_w_existence}. Then, $w \in \Gr_{\calG}^0$ is isolated and the connected component of $\Gr_{\calG}^+$ containing $w$ is open in $\Gr_{\calG}$.
	\end{lemma}
	
	\begin{proof}
		Choose a presentation of a connected component of $\Gr_{\calG}$ containing $w$ by an increasing union of $\bbG_{m,k}^{\on{pf}}$-stable perfect varieties $X$. A simple finiteness argument with reduced words, see also \cite[Theorem 2.5]{HLR18}, shows that every sufficiently large element of the Iwahori--Weyl group is bigger than $w$. Thus, we may and do assume that every irreducible component of the fiber $X_0$ over $0$ already contains the point $w$. Let $U$ be the connected component of $X^+$ containing the generic point of $X$. Since $U \subset X$ is locally closed and $X$ is irreducible, it must be an open subset. Also, its fiber $U_0$ over $0$ is non-empty, as it contains the fixed points $U^0$. As every generic point of $X_0$ specializes to $w$, we see that $U_0$ intersects $w\cdot L^{--}\calG \cdot e$ non-trivially. In particular, $U$ is contracting to $w$.
	\end{proof}
	
	We also need the following helpful criterion to detect the unibranch property.
	
	\begin{lemma}[\cite{LLHLM20}] \label{lem_LLHLM_unibranch}
		Let $X$ be a perfect $k$-variety with monoid $\bbA_{k}^{1,\on{pf}}$-action, such that $X^0=\{x\}$. Then $X$ is unibranch at $x$.
	\end{lemma}
	
	\begin{proof}
		We explain the idea of \cite[Lemma 3.4.8]{LLHLM20} for the reader's convenience. Note that the $\bbA_{k}^{1,\on{pf}}$-action extends to the normalization $Y$. The closed subspace $Y^0$ is the fiber over $x$. So now the Bia\l{}ynicki-Birula map $Y \to Y^0$ shows that the right side is connected, as $Y$ is irreducible.
	\end{proof}

	\begin{proof}[Proof of \Cref{thm_zhu_coh}]
		Thanks to \Cref{lemma_mult_action_global_grass}, we can produce for any given $S^{\on{pf}}\times \bbG_{m,k}^{\on{pf}}$-stable point $w \in \calM_{\calG,\mu}(k)$, a $\chi_w$-stable open neighborhood $\calN_w$ of $w$ in $\calM_{\calG,\mu}$ such that $\calN_w=\calN_w^+$ and $\calN_w^0=w$. Now, $\calN_w$ is irreducible, so \Cref{lem_LLHLM_unibranch} shows that $\calM_{\calG,\mu}$ is unibranch at $w$. By $L^+\calG$-equivariance, this holds at any closed point of the special fiber of $\calM_{\calG,\mu}$, hence it must be unibranch. 		
		By \Cref{lem_unibranch_reduced_sp_fib}, we conclude that the special fiber of $\calM_{\calG,\mu}^{\on{sch}}$ is reduced.
		We remind the reader that, even though we only introduce nearby cycles in \Cref{section_nearby_cycles}, they are already implicit here, since \Cref{lem_unibranch_reduced_sp_fib} relies on the calculation of the special fiber from \cite[Theorem 6.16]{AGLR22}.
		
		Now, we will not require the global rotation action, and need to work with the scheme-theoretic $\calM_{\calG,\mu}^{\mathrm{sch}}$, so we base change our entire situation to the complete local ring $k\pot{t}$ of $\bbA^1_k$. In order to see that the special fiber is Frobenius split, we must consider the flat closure $\calM_{\calG,\mu}^{\on{fl}}$ inside the ind-scheme $\Gr_{\calG}^{\on{sch}}$, i.e.~before taking perfections. In general, there is a universal homeomorphism $\calM_{\calG,\mu}^{\on{sch}} \to \calM_{\calG,\mu}^{\on{fl}}$, but it is typically only an isomorphism if $p \nmid \pi_1(G_{\mathrm{der}})$. This last condition is necessary and sufficient to avoid certain pathological non-normal Schubert varieties discovered in \cite[Theorem 2.5]{HLR18} that force $\calM_{\calG,\mu}^{\on{fl}}$ to not be normal. However, notice that $\calM_{\calG,\mu}^{\mathrm{sch}}$ does not vary under central extensions by \cite[Proposition 4.16]{AGLR22}, so we may and do assume until the end of the proof that $G_{\mathrm{der}}$ is simply connected, by passing to a z-extension.
		
		If $G_{\mathrm{der}}=G_{\mathrm{sc}}$, we claim that the natural morphism $\calM_{\calG,\mu}^{\on{sch}} \to  \calM_{\calG,\mu}^{\on{fl}}$ is an isomorphism. Since it is a finite universal homeomorphism which is an isomorphism on generic fibers by \cite[Theorem 8]{Fal03}, we may pass to special fibers and check that it becomes a closed immersion by Nakayama's lemma. Indeed, if we have a finite injection $R \to S$ of local rings and we know that $R/I \to S/IS$ is surjective for some ideal $I$ contained in the maximal ideal of $R$, then $R \to S$ is an isomorphism by Nakayama applied to $S/R$. But we know by \cite[Theorem 8.4, Proposition 9.7]{PR08} that Schubert varieties in the scheme-theoretic flag variety are compatibly Frobenius split (thanks to the equality $G_{\mathrm{der}}=G_{\mathrm{sc}}$), hence their unions are Frobenius split and thus weakly normal. In particular, the deperfected admissible locus $X:=\calA_{\calG,\mu}^{\on{sch}}$ identifies with the reduction of the special fiber $Z$ of $\calM_{\calG,\mu}^{\on{fl}}$. On the other hand, $X$ is the weak normalization of the special fiber $Y$ of $\calM_{\calG,\mu}^{\on{sch}}$ which is reduced. We get a sequence of maps \begin{equation}
		X \to Y \to Z
		\end{equation} between separated schemes whose composition is a closed immersion. Hence, the map $X \to Y$ is a closed immersion and a universal homeomorphism of reduced schemes, thus an isomorphism. We conclude that the map on special fibers $Y\to Z$ must be a closed immersion. 
	\end{proof}

	\begin{remark}
		There are two ways in which the results in \cite{Zhu14} seem to differ from \Cref{thm_zhu_coh}:
		\begin{enumerate}
			\item The original \cite[Theorem 3]{Zhu14} refers to the local model $\calM_{\calG,\mu}^{\on{fl}}$ defined via flat closure under the assumption $p \nmid \pi_1(G_{\mathrm{der}})$. During our proof of \Cref{thm_zhu_coh}, we explained this difference and how it relates to \cite[Theorem 2.5]{HLR18}.
			\item The finer \cite[Theorem 6.10]{Zhu14} asserts that the local model is compatibly Frobenius split with its special fiber. We stress that our methods do not yield this stronger claim, as we work over the perfection until we apply the normality result of \cite{PR08} in the special fiber, which is considerably weaker.
		\end{enumerate}
		
	\end{remark}
	
	\section{Combinatorics of admissible sets}
	
	In contrast with the introduction, we no longer assume that $F$ is a local field, but rather a complete discretely valued field with algebraically closed residue field $k$ of positive characteristic. Let $G$ be a connected reductive $F$-group and $\calI$ a Iwahori $O$-model of $G$.

	Fix a quadruple $(G,S,T,B)$ where $S$ is a maximal $F$-split torus whose apartment contains the alcove fixed by $\calI$, $T$ is the maximal torus given as $Z_G(S)$, and $B$ is some $F$-Borel containing $S$. We denote $N=N_G(S)$ and let $\tilde{W}=N(F)/\calT(O)$ be the Iwahori--Weyl group of $G$. It is an extension of $W_0=N(F)/T(F)$ by $T(F)/\calT(O)$. 
	We consider the Kottwitz map $T(F)/\calT(O) \to X_*(T)_I$ to the group of inertia coinvariants of $T$-coweights, inducing a bijection with inverse $\bar \nu  \mapsto t_{\bar \nu}$. Note that, according to the sign conventions of \cite{BT72,BT84}, $t_{\bar \nu}$ acts on the apartment of $S$ by $-\bar\nu$.

	Recall that the Witt flag variety $\Fl_\calI$ of \cite{Zhu17} is an ind-(perfect scheme) by \cite{BS17} stratified in $L^+\calI$-orbits indexed by $\tilde{W}$. The reduction of the local model $\calM_{\calI,\mu}$ embeds in $\Fl_\calI$ and it was shown in \cite[Theorem 6.16]{AGLR22} that it coincides with the $\mu$-admissible locus $\calA_{\calI,\mu}$. This perfect $L^+\calI$-stable subscheme is defined via the $\mu$-admissible set ${\rm Adm}(\mu)$ of Kottwitz--Rapoport \cite{KR00}: the lower poset generated by $t_{\bar \nu}$ with $\bar \nu$ running over the $\hat{T}^I$-weights $\Omega(\mu)$ of the highest weight $\hat{G}$-representation $V_\mu$. Its maximal elements form a $W_0$-orbit \cite[Theorem C]{Hai18} and we denote by $\Lambda(\mu) \subset X_*(T)_I$ its image under the Kottwitz map. 
	
	\begin{lemma}[\cite{Hai04}]\label{prop_irred_comp_codim_1}
		Suppose $x \in {\rm Adm}(\mu)$ has colength 1. Then, there are at most two distinct $\bar\nu_i \in \La(\mu)$ such that $x \leq t_{\bar\nu_i}$ for $i = 1, 2$.
	\end{lemma}
	\begin{proof}
		We proceed as in the last paragraph of \cite[Proposition 8.7]{Hai04}. There is an affine reflection $s_i$, such that $x=s_it_{\bar\nu_i}$ due to the colength hypothesis. Mapping $x$ to the group $W_0$ of euclidean transformations, we see that the fixed hyperplanes of the $s_i$ must all be parallel, so the $\bar{ \nu}_i$ lie in a $\bbR$-line. But they all have the same length by maximality, and a $\bbR$-line cannot cross a $\bbR$-sphere more than twice.
	\end{proof}
	
	The next remark was explained to us by Haines. Although it will not be needed, we leave it for the interested reader.
	
	\begin{remark}[Haines] \label{rem_irred_x_explicit}\cite[Proposition 8.7]{Hai04} can be adapted to describe the set ${\rm Irr}(x)\subset \La(\mu)$ whose translations lie above $x$ with colength $1$.  Write $x = t_{\bar\nu} s_\beta$ for some $\bar\nu \in \Omega(\mu)$ and some positive root $\beta $ of the échelonnage root system $\Sigma$, see \cite[1.4.1]{BT72}.
		\begin{enumerate}
			\item[(a)] If $x < s_\beta x$, then ${\rm Irr}(x) = \{ \bar\nu, s_\beta(\bar\nu)\}$.
			\item[(b)] If $s_\beta x < x$, then ${\rm Irr}(x) = \{ \bar\nu + \beta^\vee, s_\beta(\bar\nu + \beta^\vee)\}$. If $\mu$ is minuscule with respect to $\Sigma$, this does not occur.
		\end{enumerate}
	\end{remark}

	Another important tool for us will be the $S_2$ property of Serre for $\mu$-admissible loci. 	
	We say that a perfect $k$-scheme $X$ perfectly of finite presentation is perfectly $S_2$ if it is the perfection of an $S_2$ finite type scheme. 
	
	\begin{lemma}\label{lem_S2ification}
		If $X$ is equidimensional, there is up to isomorphism a unique perfectly finite birational morphism $X^{S_2} \to X$ such that $X^{S_2}$ is perfectly $S_2$ and identifies with the right side away from codimension $2$. In particular, $X$ is perfectly $S_2$ if and only if some (equiv. every) weakly normal finite type deperfection $X_0$ is $S_2$.
	\end{lemma}
	
	\begin{proof}
		Take a finite type reduced deperfection $X_0$ of $X$, an $S_2$ open subset $U_0 \subset X_0$ with complement of codimension $2$ and consider the $S_2$-ification $X_0^{S_2} \to X_0$ given as the normalization of $U_0 \to X_0$, see \cite[Lemma 2.11, Corollary 2.14]{Ces21}. Passing to the perfection, we get the desired morphism with the stated property. Indeed, given a finite universal homeomorphism $X_1 \to X_0$ of reduced schemes such that the preimage $U_1 \to U_0$ is also $S_2$, the local sections of $\calO_{U_1}$ are iterated $p$-th roots of those of $\calO_{U_0}$, so the same holds for the integral closures. For the last claim, observe that $X$ being perfectly $S_2$ implies $X_0^{S_2}\to X_0$ is a finite universal homeomorphism and birational, hence an isomorphism by weak normality of $X_0$.
	\end{proof}
	
	In the next result, we are going to consider the following subset \begin{equation}\mathrm{Codim}_{\leq 1}(x)=\{y \mid y \text{ has colength} \leq 1 \text{ in } {\rm Adm}(\mu)\text{ and }  y \geq x\}
	\end{equation}
	of the $\mu$-admissible set $\mathrm{Adm}(\mu)$. Recall that this is endowed with the Bruhat order, which in turn defines the so-called Bruhat graph by specifying edges between the vertices.
	
	\begin{proposition}\label{prop_S2_criterion_admissible}
		The following are equivalent: 
		\begin{enumerate}
			\item The $\mu$-admissible locus $\calA_{\calI,\mu}$ is perfectly $S_2$.
			\item For any $x \in \mathrm{Adm}(\mu)$, the (undirected) Bruhat graph of $\mathrm{Codim}_{\leq 1}(x)$ is connected.
		\end{enumerate}
	\end{proposition}
	
	\begin{proof}
		Let $\calA_{\calI,\mu}^{\on{sch}}$ be the canonical deperfection in the sense of \cite[Definition 3.14]{AGLR22}, which is weakly normal by construction. We know by \Cref{lem_S2ification} that this scheme is $S_2$ if and only if $\calA_{\calI,\mu}$ is perfectly $S_2$. Consider the Hochster--Huneke graph $\on{HH}(x)$ of the local ring of $\calA_{\calI,\mu}^{\on{sch}}$ at some closed point in the $x$-stratum, see \cite[Definition 3.4]{HH94b}. Recall that the vertices of $\on{HH}(x)$ are enumerated by the irreducible components of the local ring, and the edges connecting two of those by prime divisors contained in their intersection. Hence, $\on{HH}(x)$ and $\on{Codim}_{\leq 1}(x)$ have the same number of connected components. Now, as the irreducible components of $\calA_{\calI,\mu}^{\on{sch}}$ are unibranch by \cite[Proposition 3.7]{AGLR22}, the Hochster--Huneke graph does not change under completion in our situation, so we apply \cite[Theorem 3.6]{HH94b}, which states that the fibers of closed points in the $x$-stratum of the $S_2$-ification of $\calA_{\calI,\mu}^{\on{sch}}$ are singletons exactly when $\on{HH}(x)$ is connected. Here, we also use that $S_2$-ifications commute with completion, see \cite[Proposition 3.8]{HH94b}.
		
		At this point, we have concluded that if $\calA_{\calI,\mu}$ is perfectly $S_2$, then the graphs $\on{Codim}_{\leq 1}(x)$ are connected for all $x$. Conversely, if the graphs are all connected, then the $S_2$-ification of $\calA_{\calI,\mu}$ is bijective at $k$-valued points. Since our situation does not change when replacing $k$ by a larger algebraically closed field, the $S_2$-ification is a universally bijective perfectly finite map, and hence an isomorphism.
	\end{proof}

	It would be interesting to find a purely combinatorial proof of the $S_2$ property of admissible sets. Instead, we apply the previous criterion twice to reduce to positive characteristic geometry.

	\begin{corollary}\label{cor_adm_locus_S2}
		The $\mu$-admissible locus $\calA_{\calI,\mu}$ is perfectly $S_2$.
	\end{corollary}
	\begin{proof}
		We may and do assume that $G$ is adjoint. By \Cref{prop_S2_criterion_admissible}, it suffices to prove an analogous combinatorial property for $\on{Adm}(\mu)$. Observe that $\tilde{W}$ embeds into the group of affine transformations of the échelonnage root system $\Sigma$, its quotient $W_0$ identifies with vector transformations of $\Sigma$, and $\on{Adm}(\mu)$ is the lower poset generated by the $W_0$-orbit of some translation element, thanks to \cite[Theorem C]{Hai18}. But this can be realized as $\on{Adm}(\mu')$ for a split $k\rpot{t}$-group $G'$ by the classification of affine Coxeter groups, see \cite[1.3.17, 1.4.6]{BT72}. Applying again \Cref{prop_S2_criterion_admissible}, we need to show that, under the previous restrictions, the local rings of $\calA_{\calI',\mu'}$ are connected away from codimension $2$. But this is a consequence of \Cref{thm_zhu_coh} and Grothendieck's connectedness theorem, see \cite[Exposé XIII, Théorème 2.1]{GR05}.
	\end{proof}
	\section{Comparison of nearby cycles}\label{section_nearby_cycles}
	In this section, we compare the analytic nearby cycles of \cite{Sch17} to the formal ones of \cite[Remark 4.29]{Gle22} under an algebraicity assumption. In the classical setting of formal schemes, our results are known by the work of Huber \cite{Hub96}, so the scheme-theoretically inclined reader can safely skip this section.
	
	Fix a rank $1$ valuation ring $O$ and a prekimberlite $X$ over $\Spd O$ in the sense of \cite[Definition 4.15]{Gle22}. Let $j:Y=X^\mathrm{an}\to X$,  $i:Z^\diamond\to X$ be the natural inclusions, where $Z=X^{\on{red}}$. Given an $\ell$-torsion coefficient ring $\Lambda$ with $\ell \neq p$, the first author defines in \cite{Gle22} a naive nearby cycles functor \begin{equation}R\Psi':D_\mathrm{\acute{e}t}(Y,\Lambda)\to D_\mathrm{\acute{e}t}(Z,\Lambda).\end{equation} This arises as the (left-completion of the) derived pushforward of a morphism of sites $\Psi' \colon Y_\mathrm{\acute{e}t}\to Z_\mathrm{\acute{e}t}$ induced by canonical liftings of \'{e}tale neighborhoods \cite[Theorem 4.27]{Gle22}. On the other hand, we have the nearby cycles functor of Scholze's theory $R\Psi:D_\mathrm{\acute{e}t}(Y,\Lambda)\to D_\mathrm{\acute{e}t}(Z^\diamond,\Lambda)$ given by $R\Psi=i^*Rj_*$. 
	
	In order to compare them, we can consider the fully faithful functor (\Cref{fullyfaithfulnesssmalldia}):
	\begin{equation}
	D_\mathrm{\acute{e}t}(Z,\Lambda) \xrightarrow{c_Z^*} D_\mathrm{\acute{e}t}(Z^\diamondsuit,\Lambda) \xrightarrow{t_Z^*} D_\mathrm{\acute{e}t}(Z^\diamond,\Lambda)
	\end{equation}
	where $c_{Z}^*$ is the functor from \cite[\S 27]{Sch17} and $t_Z:Z^\diamond\to Z^\diamondsuit$ is the natural inclusion. We refer the reader to \cite[Definition 2.10]{AGLR22} for the construction of the functors $Z \mapsto Z^\diamond$ and $Z\mapsto Z^\diamondsuit$. 
	By construction, $R\Psi'$ lands in the full subcategory of $D_\mathrm{\acute{e}t}(Z^\diamond,\Lambda)$ just described, while in general $R\Psi$ might not. In this section, we prove that under some conditions if $R\Psi(A)\in D_\mathrm{\acute{e}t}(Z,\Lambda)$ then $R\Psi(A)=t_Z^*c_Z^*R\Psi'(A)$. We do this in a series of lemmas. 
	\begin{lemma}
		\label{oneglobalsectiontorulethemall}
		Let $U$ be a perfect scheme separated over $\bbF_p$. When $U$ is affine denote by $t_{U^\dagger}\colon U^\dagger\to U^\diamondsuit$ the inclusion of the closure of $U^\diamond$ in $U^\diamondsuit$.
		For every $B\in D_\mathrm{\acute{e}t}(U,\Lambda)$ we have canonical identifications $R\Gamma(U^\diamond, t_U^*c_U^*B)=R\Gamma(U^\diamondsuit, c_U^*B)=R\Gamma(U,B)$. Moreover, if $U$ is affine $R\Gamma(U^\dagger, t_{U^\dagger}^*c_U^*B)=R\Gamma(U^\diamond, t_U^*c_U^*B)$.
	\end{lemma}
	\begin{proof}
		The second equality is \cite[Proposition 27.2]{Sch17}. The first equality is done following the same reduction steps in the proof loc. cit. reducing to the case where $U=\Spec(\bbA^I)$ for a set $I$. This case is settled during the proof of \cite[Proposition 27.2]{Sch17} by appealing to the invariance under change of base field \cite[Theorem 19.5]{Sch17} and comparing cohomology on affine space $\bbA^I_C$ and the unit ball $\bbB^I_C$. 
		
		If $i:U^\diamond\to U^\dagger$ is the natural inclusion, then a Postnikov tower argument, quasicompact basechange, \cite[Proposition 17.6]{Sch17}, and that $c_U^*B$ is overconvergent imply that $Ri_*t_{U^\diamond}^*c_U^*B=t_{U^\dagger}^*c_U^*B$. This proves the second claim. 
	\end{proof}
	\begin{proposition}
		\label{fullyfaithfulnesssmalldia}
		Let $U$ be a perfect scheme separated over $\bbF_p$, the functor $t_U^*c_U^*$ is fully faithful. 
	\end{proposition}
	\begin{proof}
		We may assume $A\in D^+_\mathrm{\acute{e}t}(U,\Lambda)$, since full faithfulness extends formally to left-completions. 
		By descent, we may also assume $U$ is defined over $\Spec(\overline{\bbF}_p)$.
		For $A\in D_\mathrm{\acute{e}t}(U,\Lambda)$ we verify $A\to Rc_{U,*}Rt_{U,*}t_U^*c_U^*A$ is an isomorphism by checking this on sections. Let $B=c_U^*A$, we also denote by $B$ the evident pullback to different loci. 
		By \Cref{oneglobalsectiontorulethemall} one reduces to proving $R\Gamma(Q,B)=R\Gamma(V^\diamond,B)$ for sufficiently small \'etale neighborhoods $V\to U$, for $Q=U^\diamond\times_{U^\diamondsuit} V^\diamondsuit$. 
		Let $W\subseteq U$ be an open subset and let $V_W=V\times_U W$, observe that $W^\diamond\times_{U^\diamond} Q=W^\diamond\times_{W^\diamondsuit} V^\diamondsuit_W$. 
		Applying descent to an open cover $\coprod W\to U$ and shrinking $V$ we may assume that $U$ and $V$ are affine. 
		Let $\overline{V}$ be the closure of $V^\diamond$ in $Q$, arguing as in \Cref{oneglobalsectiontorulethemall} we have $R\Gamma(V^\diamond,B)=R\Gamma(\overline{V},B)$ so it suffices to see that $R\Gamma(Q,j_!B)=0$ for the open immersion $j:Q\setminus \overline{V}\to Q$.
		By Noetherian approximation $V\to U$ arises as the base change of an \'etale map $S\to T$ with $S$ and $T$ the spectrum of finite type $\overline{\bbF}_p$-algebras. 
		All spaces above come from basechange by the map $\pi:U^\diamond\to T^\diamondsuit$ which is qcqs. 
		Indeed, $Q$ corresponds to $S^\diamondsuit$, $V^\diamond$ corresponds to $S^\diamond$ and $\overline{V}$ corresponds to $S^\dagger$. 
		To see the last identification recall that $U^\dagger\times_{T^\dagger} S^\dagger=V^\dagger$ as it is evident from the moduli description \cite[Proposition 2.22]{Gle22}, and that $V^\diamond$ is always dense in $V^\dagger$ \cite[Proposition 2.24]{Gle22}.
		By \cite[Proposition 17.6]{Sch17}, it suffices to prove $R\Gamma(S^\diamondsuit, Rs_!s^*R\pi_*B)=0$ where $s:S^\diamondsuit\setminus S^\dagger\to S^\diamondsuit$ is the natural inclusion. 
		We claim $R\Gamma(S^\diamondsuit, Rs_!K)=0$ for any $K$.
		Finding a closed immersion $S\to \bbA_{\overline{\bbF}_p}^N$, it suffices to prove this if $S=\bbA_{\overline{\bbF}_p}^N$. 
		This follows from \cite[Theorem IV.5.3]{FS21}, since $S^\diamondsuit \setminus S^\dagger$ is a spatial diamond partially proper over $\Spd(\overline{\bbF}_p)$. Indeed, it is the pointed formal completion of the divisor at infinity in $\bbP_{\overline{\bbF}_p}^N$.
	\end{proof}
	\begin{lemma}
		\label{vanishingofpartiallycompact}
		Let $A\in D_\mathrm{\acute{e}t}(Y,\Lambda)$. Then $R\Gamma(X,Rj_!A)=0$, or equivalently $R\Gamma(Y,A)=R\Gamma(Z^\diamond,R\Psi(A))$.
	\end{lemma}
	\begin{proof}
		Since $j_!$ commutes with canonical truncations a Postnikov limit argument allows us to assume $A\in D^+_\mathrm{\acute{e}t}(Y,\Lambda)$. Since $X$ is a specializing v-sheaf there is a hypercover of $X_\bullet\to X$ where each $X_i$ is of the form $\coprod_{j\in I_i}\Spd(R^+_j)$ for $I_j$ a set and the $\Spa(R_j,R_j^+)$ are strictly totally disconnected perfectoid spaces. By v-hyperdescent \cite[Proposition 17.3]{Sch17} and proper basechange we may assume $X=\Spd(R^+)$. At this point we may cite \cite[Remark V.4.3]{FS21}. Let us explain a detail. A choice of pseudouniformizer defines a qcqs map $\Spd(R^+)\to \Spd(W(k)\pot{t})$ and \cite[Proposition 17.6]{Sch17} reduces the computation to the case $X=\Spd(W(k)\pot{t})$ which follows from \cite[Theorem IV.5.3]{FS21}.
	\end{proof}
	We also have a map of sites $f \colon X_v\to Z_\mathrm{\acute{e}t}$ induced again by canonical liftings of \'{e}tale neighborhoods. Restricted to $Z^\diamond_v$, $f\circ i_v$ sends an \'{e}tale map $U\to Z$ to $U^\diamond\to Z^\diamond$. In the case of $Y_v$, $f\circ j_v$ factors as the composition of $\nu_Y \colon Y_v\rightarrow Y_\mathrm{\acute{e}t}$ with $\Psi'\colon Y_{\mathrm{\acute{e}t}} \to Z_{\mathrm{\acute{e}t}}$. We get a functor $Rf_{v,*}:D(X_v,\Lambda)\to D_\mathrm{\acute{e}t}(Z,\Lambda)$ that factors through $D_\mathrm{\acute{e}t}(X,\Lambda)$, since its right adjoint factors through the inclusion $D_\mathrm{\acute{e}t}(X,\Lambda)\subseteq D(X_v,\Lambda)$. We denote by $Rf_*$ the induced map on this latter category. 
	\begin{lemma}
		\label{3nearbycycles}
		Let $A\in D_\mathrm{\acute{e}t}(Y,\Lambda)$.
		Then $R(f\circ i)_*R\Psi(A)=R\Psi'A$.
	\end{lemma}
	\begin{proof}
		We have an exact triangle $Rj_!
		A\to Rj_*A\to i_*R\Psi(A)$ in $D_\mathrm{\acute{e}t}(X,\Lambda)$, to which we apply $Rf_*$. 
		Now, by \cite[Proposition 14.10, Proposition 14.11]{Sch17} $Rf_* Rj_*=R(f_v\circ j_v)_* \nu_{Y}^* =R\Psi' \nu_{Y,*} \nu_{Y}^*=R\Psi'$. 
		It suffices to prove $Rf_*j_!A=0$. It suffices to prove $R\Gamma(U,Rf_*j_!A)=0$ for all $U\in Z_\mathrm{\acute{e}t}$. 
		We compute directly $R\Gamma(U,Rf_*j_!A)=R\Gamma(\widehat{X}_U,j_!A)=0$. The last equality follows from \Cref{vanishingofpartiallycompact}, and the fact that passing to \'etale formal neighborhoods preserves being a prekimberlite. 
	\end{proof}
	
	\begin{lemma}
		\label{nearbycyclesofperfect}
		We have canonical identifications $R(f\circ i)_*t_Z^*c_{Z}^*A=Rc_{Z,*}c_{Z}^*A=A$.
	\end{lemma}
	\begin{proof}
		The second one is \cite[Proposition 27.2]{Sch17}. Let $h\colon U\to Z$ any \'etale neighborhood and let $t_U:U^\diamond\to U^\diamondsuit$ denote the natural map. We compute as follows:
		\begin{equation}
		\begin{split}
		R\Gamma(U,h^*Rc_{Z,*}c_{Z}^*A) & =  R\Gamma(U^\diamondsuit,h^{\diamondsuit,*}c_{Z}^*A) \\ 
		& =  R\Gamma(U^\diamondsuit,c_U^*h^*A) \\  
		& =  R\Gamma(U^\diamond,t_U^*h^{\diamondsuit,*}c_{Z}^*A) \\ 
		& =  R\Gamma(U^\diamond,h^{\diamond,*}t_Z^*c_{Z}^*A) \\ 
		& =  R\Gamma(U,h^*R(f\circ i)_*t^*c_{Z}^*A),			\end{split}\end{equation}
		where we have applied \cite[Proposition 27.1]{Sch17} twice to commute $c^*$ and $h^*$, and also \Cref{oneglobalsectiontorulethemall} in the middle equality.\end{proof}

	From now on assume $Y$ is a spatial diamond that has finite cohomological dimension as in \cite[Proposition 20.10]{Sch17} so that $D(Y_\mathrm{\acute{e}t},\Lambda)=D_\mathrm{\acute{e}t}(Y,\Lambda)$. We also assume that $Z$ is perfectly of finite type over the residue field of $O$ so that $D(Z_\mathrm{\acute{e}t},\Lambda)=D_\mathrm{\acute{e}t}(Z,\Lambda)$. With these hypothesis $R\Psi':D_\mathrm{\acute{e}t}(Y,\Lambda)\to D_\mathrm{\acute{e}t}(Z,\Lambda)$ is defined site theoretically (without having to left-complete) and taking stalks is well-behaved.
	\begin{definition}
		\label{defialgebraicity}
		Identify $D_\mathrm{\acute{e}t}(Z,\Lambda)$ with its essential image in $D_\mathrm{\acute{e}t}(Z^\diamond,\Lambda)$ under $t_Z^*c_{Z}^*$. We say $K\in D_\mathrm{\acute{e}t}(Z^\diamond,\Lambda)$ is algebraic if $K\in D_\mathrm{\acute{e}t}(Z,\Lambda)$.
	\end{definition}
	
	Recall that for a closed subset $S \subset U$ and an étale map $U\to Z$, we denote by $\widehat{X}_{/S}$ the formal neighborhood at $S$ in the sense of \cite[Definition 4.18]{Gle22} and by $X^\circledcirc_{/S}:=\widehat{X}_{/S}\times_X X^{\mathrm{an}}$ the tubular neighborhood in the sense of \cite[Definition 4.38]{Gle22}. (The references given assume $U=Z$, but it extends to étale neighborhoods by \cite[Theorem 4.27]{Gle22}.)
	
	\begin{theorem}
		\label{thmnearbycycles}
		\label{computationofstalks}
		Let the notation be as above and $A\in D_\mathrm{\acute{e}t}(Y)$. The following hold:
		\begin{enumerate}
			\item If $\overline{x}\in Z$ is a closed point, then $(R\Psi A)_{\overline{x}}=R\Gamma(X^\circledcirc_{/\overline{x}},A)$.
			\item $(R\Psi' A)_{\overline{x}}=\varinjlim_{\overline{x}\in U}R\Gamma(X^\circledcirc_{/U},A)$, where $U$ runs over all étale neighborhoods of $Z$ at $\overline{x}$.
			\item We have that $R\Psi A=R\Psi' A$ and $R\Gamma(X^\circledcirc_{/\overline{x}},A)=\varinjlim_{\overline{x}\in U}R\Gamma(X^\circledcirc_{/U},A)$ hold when $R\Psi A$ is algebraic.
		\end{enumerate}
	\end{theorem}
	\begin{proof}
		The third claim follows directly from \Cref{nearbycyclesofperfect} and \Cref{3nearbycycles}, and from the second and first claim.  
		Let $\iota_{\overline{x}}$ denote the inclusion of $\overline{x}$. Notice that $\iota_{\overline{x}}$ factors through the open immersion $\widehat{X}_{/\overline{x}}\to X$. 
		By smooth base change and \Cref{vanishingofpartiallycompact} applied to $\widehat{X}_{/\overline{x}}\to X$ we get $(R\Psi A)_{\overline{x}}=R\Gamma(X^\circledcirc_{/\overline{x}},A)$, proving the first claim. 
		Finally, by \cite[Proposition 27.1.(ii)]{Sch17} $\iota^{\diamondsuit,*}_{\overline{x}}c_Z^*(R\Psi' A)=c_{\overline{x}}^*(R\Psi' A_{\overline{x}})$ and the latter is by definition computed by $\varinjlim_{\overline{x}\in U}R\Gamma(X^\circledcirc_{/U},A)$, since it is site theoretic. 
	\end{proof}
	
	\section{Proof of unibranchness}

	In this section, we prove \Cref{thm_normal_LM}.
	During this section, we set $\La=\bbF_\ell$ and consider $\calM_{\calG,\mu}$ already after base change to $\Spd O_C$, where $C$ is a complete algebraic closure of $F$ with ring of integers $O_C$. The reader who is only interested in the scheme-theoretic local models can imagine this takes place in the realm of formal and rigid-analytic geometry.
	
	The object $\calZ_\mu=R\Psi(\on{IC}_\mu)$ allows us to read off the set of connected components of tubes.
	
	\begin{lemma}\label{lem_conn_comps_cohom}
		There is an equality
		$\#\pi_0(\calM^\circledcirc_{\calG,\mu/x})=\on{dim}_{\bbF_\ell}\calH_{x}^{-\langle 2\rho, \mu \rangle }\calZ_\mu$. 
	\end{lemma}
	
	\begin{proof}
		Let $j_\mu \colon \Gr_{G,\mu}^\circ \to \Gr_G$ denote the orbit inclusion in the generic fiber. It follows from perverse left t-exactness of $Rj_{\mu*}$ and Schubert varieties being unibranch that \begin{equation}\calH^{-\langle 2\rho, \mu \rangle }(\on{IC}_\mu)=\calH^0 (Rj_{\mu*}\bbF_\ell)=\bbF_\ell\end{equation} is a constant sheaf on the generic fiber. (Strictly speaking, the $B_\dR^+$-affine Grassmannian is not an ind-scheme, but facts such as these can be reduced via \cite[Corollary VI.6.7]{FS21} to the Witt affine Grassmannian of the split form of $G$, which is an ind-perfect scheme.) This implies that we may replace the right side of the claimed equality by $\on{dim}_{\bbF_\ell}i_{x}^*R^0j_*\bbF_\ell$, with $j \colon \Gr_{G,\mu} \to \calM_{\calG,\mu}$ the generic fiber inclusion and $i_x\colon \Spd k \to \calM_{\calG,\mu}$ the inclusion of the point $x$. But the latter equals $\#\pi_0(\calM^\circledcirc_{\calG,\mu/x})$ by \Cref{computationofstalks}.
	\end{proof}

	Thanks to \cite[Lemma 5.26]{Gle22}, we can reduce the proof of \Cref{thm_normal_LM} to the case when $\calG$ is Iwahori, provided we verify the following.
	
	\begin{lemma}\label{lem_iwahori_conn_fibers}
		If $\calI\to \calG$ is a Iwahori dilation, then the geometric fibers of $\pi \colon \calA_{\calI,\mu}\to \calA_{\calG,\mu}$ are connected.
	\end{lemma}
	
	\begin{proof}
		By $\calI$-equivariance, we are reduced to considering the fiber over the image $w_\calG$ of a $L^+\calT$-fixed point $w \in \calA_{\calI,\mu}$. We may and do assume that $w$ is minimal in its right $W_\calG$-coset. Using Demazure resolutions, one sees that the intersection of $\pi^{-1}(w_\calG)$ with any Schubert variety is connected. As all of those subschemes must contain $w$ by minimality, the fiber itself must be connected. 
	\end{proof}

	From now on, we work with a Iwahori model $\calI$. To calculate in codimension $1$, we are going to apply the Wakimoto filtration of the Gaitsgory central functor studied in \cite{Gai01,AB09} in equicharacteristic and in \cite{AGLR22,ALWY22} in mixed characteristic.
	
	\begin{theorem}[\cite{ALWY22}]\label{theorem_perverse_wakimoto_filtration}
		The functor $R\Psi$ is perverse t-exact. Moreover, the perverse sheaf $R\Psi(\on{Sat}(V))$ admits a filtration with subquotients isomorphic to $\calJ_{\bar{\nu}}\otimes V(w_0\bar{\nu})$.
	\end{theorem}

		The Wakimoto sheaves $\calJ_{\bar{\nu}}$ depend crucially on the choice of a Borel subgroup $B \subset G$ and we exploit this degree of freedom in \Cref{prop_unibranch_codim_1} by conveniently choosing $B$. For the images $\bar{\nu}$ of $B$-dominant coweights, $\calJ_{\bar{\nu}}$ are defined as the costandard object $\nabla_{\bar{\nu}}=Rj_{\bar{\nu},*}\bbF_\ell[\ell(t_{\bar{\nu}})]$, which admits the standard object $\Delta_{\bar{\nu}}=j_{\bar{\nu},!}\bbF_\ell[\ell(t_{\bar{\nu}})]$ as a left and right inverse for convolution. Here $j_{\bar{\nu}}: \Fl_{\calI,t_{\bar{\nu}}}^\circ \to \Fl_{\calI}$ is the natural orbit inclusion, which is affine, so the derived functors are perverse t-exact by \cite[Corollaire 5.1.3]{BBDG18}. The definition extends by convolution and linearity to the other $\bar{\nu}$. Just as in \cite[Theorem 5]{AB09}, it turns out that $\calJ_{\bar{\nu}}$ is a perverse sheaf concentrated in $\Fl_{\calI,t_{\bar{\nu}}}$ and restricts to $\bbF_\ell[\ell(t_{\bar{\nu}})]$ on $\Fl_{\calI,t_{\bar{\nu}}}^\circ$.		
		The proof of \Cref{theorem_perverse_wakimoto_filtration} in \cite{ALWY22} relies not only on the $\calJ_{\bar{\nu}}$, but also on $R\Psi(\on{Sat}(V))$ being central, see \cite[Proposition 6.17]{AGLR22}, and the computation of its constant terms, see \cite[Equation (6.32)]{AGLR22}.

	\begin{lemma}\label{lemma_cohomology_wakimoto}
		For any $x\in\Fl_\calI(k)$, the $\bbF_\ell$-vector space $\calH_x^{-\ell(t_{\bar{\nu}})}\calJ_{\bar{\nu}}$ is either zero or one-dimensional.
	\end{lemma}
	
	\begin{proof}
		Consider the adjunction map $\calJ_{\bar{\nu}} \to \nabla_{\bar{\nu}}$. It follows that the kernel and cokernel have support contained in $\Fl_{\calI,t_{\bar{\nu}}} \setminus \Fl_{\calI,t_{\bar{\nu}}}^\circ$. In particular, they are concentrated on degrees strictly larger than $-\ell(t_{\bar{\nu}})$ by perversity, so that
		\begin{equation}
		\calH_x^{-\ell(t_{\bar{\nu}})}\on{ker}(\calJ_{\bar{\nu}} \to \nabla_{\bar{\nu}} )=\calH_x^{-\ell(t_{\bar{\nu}})}\on{coker}(\calJ_{\bar{\nu}} \to \nabla_{\bar{\nu}} )=0.
		\end{equation}
		This yields an inclusion $\calH_x^{-\ell(t_{\bar{\nu}})}\calJ_{\bar{\nu}} \subset \calH_x^{-\ell(t_{\bar{\nu}})} \nabla_{\bar{\nu}} $. The latter sheaf is constant equal to $\bbF_\ell$, because the Schubert perfect variety $\Fl_{\calI,t_{\bar{\nu}}}$ is normal.
	\end{proof}
	
	\begin{proposition}\label{prop_unibranch_codim_1}
		Given $x \in \calA_{\calI,\mu}(k)$ whose $L^+\calI$-orbit has codimension at most $1$, we have an equality $\calH_x^{-\langle2\rho, \mu \rangle}\calZ_\mu=\bbF_\ell$.
	\end{proposition}
	
	\begin{proof}
		This is evident in codimension $0$. 
		By abuse of notation, we denote by $x$ the corresponding element in $\on{Adm}(\mu)$ and let $x < t_{\bar{\nu}_i}$ with $\bar \nu_i \in \La(\mu)$ with $i=1,2$ be the only (possibly equal) maximal elements above $x$ by \Cref{prop_irred_comp_codim_1}.
		Assume without loss of generality that $\bar{\nu}_1\in X_*(T)_I^-$ by replacing our initial choice of Borel if necessary. 
		Consider the Wakimoto filtration of the perverse sheaf $\calZ_\mu$ given in \Cref{theorem_perverse_wakimoto_filtration}. 
		Only the $\calJ_{\bar{\nu}_i}\otimes V(w_0\bar{\nu}_i)$ for $i=1,2$ contribute to the stalk at $x$.
		Since we chose $\calJ_{\bar{\nu}_1}$ to be standard, it is concentrated exclusively on a maximal orbit and hence does not contribute to the stalk at $x$. 
		Consequently, $\calH_x^{-\langle2\rho, \mu \rangle}\calZ_\mu=V(w_0\bar{\nu}_2)\otimes \calH_x^{-\langle2\rho, \mu \rangle}(\calJ_{\bar{\nu}_2})$.
		As $V_\mu(w_0\bar{\nu}_2)=\bbF_\ell$ for extremal weights, \Cref{lemma_cohomology_wakimoto} bounds the dimension above by $1$. 
		On the other hand, \Cref{lem_conn_comps_cohom} and density of the generic fiber bound the dimension below by $1$. (This also implies $\bar \nu_1 \neq \bar \nu_2$.)
	\end{proof}

	We have now all the necessary tools at our disposal to finish the proof of the main theorem.

	\begin{proof}[Proof of \Cref{thm_normal_LM}]
		Assume there is $x \in \calA_{\calI,\mu}(k)$ such that $\calM_{\calI,\mu/x}^\circledcirc$ is disconnected. By \Cref{computationofstalks} there is a connected \'etale neighborhood $x\in U\to \calA_{\calI,\mu}$ such that \begin{equation}{\calM}^{\circledcirc}_{\calI,\mu/U}=V_1 \sqcup V_2\end{equation} is a union of two non-empty clopen subsets. 
		Let $\iota \colon W \subset U$ be the open subset obtained by pulling back the union of $L^+\calI$-orbits of $ \calA_{\calI,\mu}$ of codimension at most $1$. 
		Since the $S_2$ property is stable under étale maps and by \Cref{cor_adm_locus_S2}, the scheme $W$ is connected. Indeed, otherwise $\calO_U=\iota_*\calO_W$ would contain non-trivial idempotents contradicting our assumption on $U$. Apply \Cref{lem_conn_comps_cohom}, \Cref{prop_unibranch_codim_1} and \cite[Lemma 4.55]{Gle22} to conclude that ${\calM}^{\circledcirc}_{\calI,\mu/W}$ is connected as well, so $\on{sp}^{-1}(W)$ is entirely contained in $V_1$, say.
		
		On the other hand, the disjoint union induces a direct sum decomposition $\calZ_\mu|_U=A_1 \oplus A_2$ into non-zero perverse sheaves, where the $A_i$ are the nearby cycles of $\on{IC}_\mu$ after pulling it back to $V_i$. We know by construction that $\calH^{-\langle 2\rho, \mu \rangle}A_2$ does not vanish, and also that the support of $A_2$ has codimension at least $2$ in $U$ by the previous paragraph. These two facts contradict $A_2$ being perverse, so our initial assumption was wrong. 
	\end{proof}
	\bibliography{biblio.bib}
	\bibliographystyle{alpha}
\end{document}